\documentclass[reqno]{amsart}
\usepackage{pdfsync}
\usepackage{amsfonts,amssymb}
\usepackage{verbatim,graphicx,enumerate,xspace,url}
\usepackage[colorlinks=true, pdfstartview=FitV, linkcolor=blue,
 citecolor=blue, urlcolor=red]{hyperref}
\usepackage[author-year,initials,nobysame,backrefs]{amsrefs}

\usepackage{xr-hyper}
\usepackage{hyperref}

\renewcommand{\P}{{\mathbf P}}
\newcommand{\E}{{\mathbf E}}
\newcommand{\Z}{{\mathbf Z}}
\newcommand{\tmix}{t_{{\rm mix}}}
\newcommand{\ep}{\varepsilon}
\newcommand{\st}{\,:\,}
\newcommand{\X}{{\mathcal X}}
\newcommand{\one}{{\boldsymbol 1}}
\newcommand{\R}{{\mathbf R}}
\DeclareMathOperator{\var}{Var}

\newcommand{\deq}{\stackrel{def}{=}}

\newtheorem{thm}{Theorem}%[chapter]

\newtheorem{prop}[thm]{Proposition}

\newtheorem{lem}[thm]{Lemma}

\newtheorem*{thm*}{Theorem}
\newtheorem*{cor*}{Corollary}

\theoremstyle{definition}

\theoremstyle{remark}
\newtheorem{rmk}[thm]{Remark}

\newtheorem*{fact*}{Fact}
\newtheorem*{not*}{Notation}
\newtheorem*{claim*}{Claim}

\newenvironment{enumeratei}{\begin{enumerate}[\upshape (i)]}
                           {\end{enumerate}}

\title{Mixing of the exclusion process with
small bias}
 \author{David A. Levin}
 \address{Department of Mathematics, Fenton Hall, University of Oregon 1222,
 Eugene, OR, 97403-1222.}
 \author{Yuval Peres}
 \address{Microsoft Research, 1 Microsoft Way, Redmond WA 98052}

\begin{document}

\maketitle

\begin{abstract}
 We analyze the mixing behavior of the biased exclusion process  on a path of length $n$
 as the bias $\beta_n$ tends to $0$ as $n \to \infty$.    We show that the sequence
 of chains has a pre-cutoff, and interpolates between the unbiased exclusion and
 the process with constant bias. As the bias increases, the
 mixing time undergoes two phase transitions: one when $\beta_n$ is of order $1/n$, and
 the other when $\beta_n$ is order $\log n/n$.
 \end{abstract}

\section{Introduction}

Suppose $k$ particles are placed on vertices of the $n$-path,
with no site multiply occupied.   The \emph{biased exclusion process}
is the Markov chain $(X_t)_{t \geq 0}$ with transitions as follows:
\begin{itemize}
  \item choose uniformly among the $n-1$ edges of the path, 
  \item if both vertices of the selected edge are either occupied or
    unoccupied, do nothing,
  \item if there is exactly one particle on the edge, place it on the
    right vertex with probability $p = (1+\beta)/2$ and on the left
    with probability $q = (1-\beta)/2$.
 \end{itemize}
The canonical case is when $n$ is even and $k = n/2$.
This defines a reversible ergodic Markov chain, which
has a unique stationary distribution $\pi$.
It is natural to ask about its \emph{mixing time},
\[
\tmix(\ep) 
  =  \min\{ t \geq 0 \st 
   \max_{\sigma} \| \P_\sigma(X_t  \in \cdot) - \pi \|_{{\rm TV}} < \ep \} \,.
\]
We write $\tmix$ for $\tmix(1/4)$.
When $\beta = 0$, \ocite{W:MTS} proved
\[
\frac{1}{\pi^2}(1 + o(1)) n^3 \log n 
\leq \tmix(\ep) \leq \frac{2}{\pi^2}[1 + o(1)] n^3 \log (n/\ep) \,,
\]
and conjectured that the lower bound is sharp.
Recently, Lacoin \ycite{Lacoin} answered this, proving that
the process has a \emph{cutoff}, i.e.\ 
\[
\lim_{n \to \infty} \frac{\tmix(\ep)}{n^3 \log n} \to \frac{1}{\pi^2}\,.
\]
It is worth observing that the eigenfunction lower bound method introduced
in Wilson~\ycite{W:MTS} turns out to be widely applicable,
giving sharp lower bounds for many models.

When $\beta > 0$, the mixing time  was first studied by
\fullocite{BBHM}, who proved $\tmix = O(n^2)$. A simpler path coupling
proof was given by \fullocite{GPR}.   (This proof is repeated here as
the upper bound in Theorem \ref{thm:mixbe}.) 
The purpose of this paper is to understand the mixing behavior
when the bias may depend on $n$ and in particular when $\beta_n \to 0$
as $n \to \infty$. 
We show that in all cases, there is a \emph{pre-cutoff}, meaning that there are universal
constants $c_1 < c_2$ so that
\[
c_1 \leq \frac{\tmix(1-\ep)}{\tmix(\ep)} \leq c_2 \,.
\]
We find that, depending on the rate at which $\beta \to 0$, the mixing time
interpolates between the unbiased and constant bias cases.

Below summarizes our results.

We write $a_n \asymp b_n$ to mean that there exist constant $0 < c_1,c_2 < \infty$, not
depending on $\beta$, so that $c_1 \leq a_n/b_n \leq c_2$.
\begin{thm} \label{Thm:Main} 
  Consider the $\beta$-biased exclusion
  process on $\{1,2,\ldots,n\}$ with $k$ particles.  We assume that
  $k/n \to \rho \leq 1/2$.
  \begin{enumeratei}
  \item If $n\beta \leq 1$, then
  \begin{equation} \label{Eq:tmix1}
  	\tmix \asymp n^3 \log n \,.
  \end{equation}
  \item If $1 \leq n\beta \leq \log n$, then
  \begin{equation} \label{Eq:tmix2}
  	\tmix \asymp \frac{n \log n}{\beta^2} \,.
  \end{equation}
  \item If $n\beta > \log n$, then
  \begin{equation} \label{Eq:tmix3}
  	\tmix \asymp \frac{n^2}{\beta} \,.
  \end{equation}
  \end{enumeratei}
\end{thm}

%
%  \item  Suppose that $n \beta \to \zeta$.
%    Then
%    \[
%    \tmix(\ep) 
%      \geq \frac{n^3}{\zeta^2 + \pi^2}[1+o(1)] 
%      \Bigl( \log n   + \log[(1-\ep)/\ep] \Bigr)
%      = \Omega\Bigl(\frac{n\log n}{\beta^2} \Bigr) \,.
%    \]
%    If $\zeta > 0$, then
%    \begin{align*}
%      \tmix(\ep)
%      & \leq\frac{4n}{\beta^2}[\log n + \log(\ep^{-1}) + O(1)]  \\
%      & = \frac{4n^3}{\zeta ^2}[1+o(1)]
%          \left(\log n + \log(\ep^{-1}) + O(1)\right) \,.
%    \end{align*}
%    If $\zeta = 0$, then
%    there exists a constant $c_1 > 0$ such that 
%    \begin{equation} \label{Eq:UBsmall} 
%      \tmix(\ep) \leq c_1 n^3 \log n \,.
%    \end{equation}
%  \item
%    If $n\beta \to \infty$ and $n\beta \leq \log n$, then
%    \begin{multline*}
%      \frac{n}{\beta^2}[1 + o(1)]\Bigl(\log n + \log((1-\ep)/\ep)\Bigr) \\
%      \leq \tmix(\ep)  
%      \leq \frac{6n}{\beta^2}\Bigl[\log n [1 + o(1)] + \log(\ep^{-1}) +
%        O(1) \Bigr] \,.
%    \end{multline*}
%  \item 
%    Suppose that $c_1 (\log n)/n < \beta < c_2< 1$ for constants $c_1$ and $c_2$.
%    For any $\ep > 0$ and $\delta > 0$, if $n$ is large enough, then
%    \[
%    \tmix(\ep) \geq \frac{(1-\delta)nk}{\beta} \,,
%    \]
%    and	
%    \[
%    \tmix(\ep) \leq \frac{2n^2}{\beta}(1 + \log(\ep)/n + O(\beta)) \,.
%    \]
%  \end{enumeratei}
%\end{thm}

We provide more precise estimates on $\tmix(\ep)$ in Proposition
\ref{Prop:LB1}, Proposition \ref{Prop:LB2}, and Theorem \ref{thm:mixbe}.
In particular, the lower bound in \eqref{Eq:tmix1} follows from
Proposition \ref{Prop:LB1}, the lower bound in \eqref{Eq:tmix2} follows
from Proposition \ref{Prop:LB2}, and the lower bound in \eqref{Eq:tmix3}
follows from Proposition \ref{Prop:LBU}.  The upper bounds
in \eqref{Eq:tmix2} and \eqref{Eq:tmix3} follow from Theorem \ref{thm:mixbe},
and the upper bound in \eqref{Eq:tmix1} follows from Proposition \ref{Prop:NearUn}.

Since the behavior of the individual particles remains
diffusive in the $\beta n < 1$ regime, it is not surprising that
the mixing time has the same order as the unbiased process
in this case.  The change of the functional form of the mixing time 
at $\beta n = \log n$ is a more unexpected transition.

A path coupling gives useful upper bounds for $\beta \geq c/n$. 
When $\beta n$ is small, we use a simple
coupling adapted from a coupling for (unbiased) random adjacent
transpositions given in \ocite{A:RWG}.  In the unbiased case, $k$
coupled unbiased random walks must hit zero.  The bias introduced when
$\beta n$ is small doesn't overwhelm the diffusive motion, so the same
idea works. 

For lower bounds,  when $\beta n \leq \log n$, 
we use Wilson's method (introduced in \ocite{W:MTS}).
Thus we need the eigenfunction corresponding to the second eigenvalue,
which we explicitly compute.   When $\beta n > \log n$, we follow
the left-most particle, and show it needs at least order $n^2/\beta$ moves to mix.

The organization of the paper is as follows.  After giving definitions
in Section \ref{Sec:Defn}, in Section
\ref{Sec:SpectralLB} we  compute the eigenfunction needed for
Wilson's method, and provide the corresponding lower bounds.
In particular, the lower bounds in Theorem \ref{Thm:Main}
(i) and (ii) are given in Propositions \ref{Prop:LB1} and
\ref{Prop:LB2}, respectively.

We give the two upper bounds in Section \ref{Sec:UB}:
The upper bound in \eqref{Eq:tmix1} is given in
Proposition \ref{Prop:NearUn}, and the other upper bounds
in Theorem \ref{Thm:Main} are all immediate from
Theorem \ref{thm:mixbe}.

We conclude with the single particle lower bound needed
for Theorem \ref{Thm:Main} (iii) in Section \ref{Sec:LBU}.

\section{Definitions}
\label{Sec:Defn}
\subsection{Path description}

It will sometime be convenient to use a bijection of the state-space $\{0,1\}^n$
of the particle process to the space of nearest-neighbor paths of
length $n$ which begin at $0$ and have exactly $k$ \emph{up} increments
and $n-k$ \emph{down} increments.   For a particle configuration
$\sigma \in \{0,1\}^n$, let $h:\{0,1,\ldots,n\} \to \Z$ be defined by $h(0) = 0$, and
\[
h(j) - h(j-1) = (-1)^{1-\sigma(j)} \,,
\]
so occupied sites correspond to increments and vacant sites correspond
to decrements of the path.  See Figure \ref{Fig:ExclMov} for an illustration.

 \begin{figure}[h]
   \begin{center}
     \includegraphics[scale=0.20]{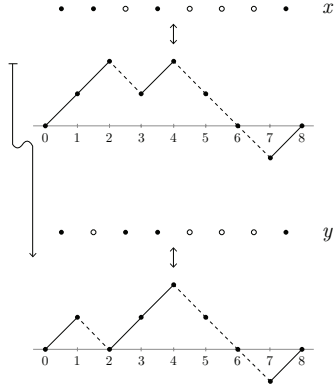}
     \caption{The correspondence between particle representation and
       path representation for neighboring configurations $x,y$.  Node
       $2$ of the path is updated in configuration $x$ to obtain $y$.
       This corresponds to exchanging the particle at vertex $2$ with
       the hole at vertex $3$. \label{Fig:ExclMov}}
   \end{center}
 \end{figure}

The dynamics on the path are as follows:  pick among the $n-1$ interval vertices
of the path.   If the path is a local extremum, refresh it with a local maximum
with probability $q$, and a local minimum with probability $p$.   If the
chosen vertex is not an extremum, do nothing.  See again Figure \ref{Fig:ExclMov} for 
an illustration of a transition, and Figure \ref{Fig:BiasedTrans} for the possible transitions from a particular path.

\begin{figure}[h]
   \begin{center}
     \includegraphics[scale=0.45]{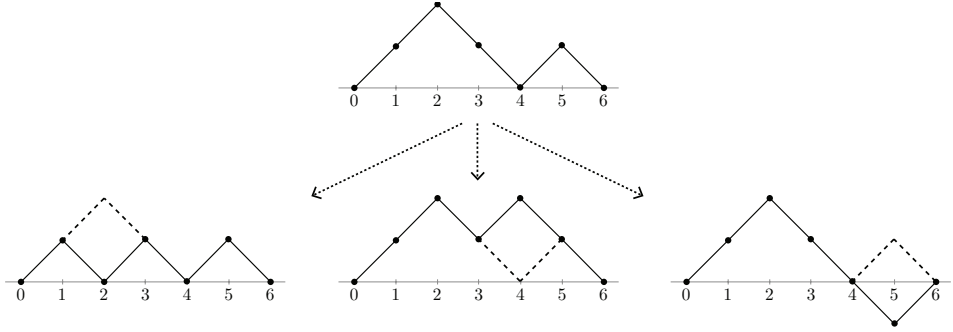}
     \caption{The possible transitions from a given configuration.
            \label{Fig:BiasedTrans}}
   \end{center}
\end{figure}

It will be convenient to move back and forth from the particle description
and the path description, and we will freely do so.

\section{Spectral Lower bounds}
\label{Sec:SpectralLB}

Here we set $\alpha = \sqrt{p/q}$; our assumption is always that 
$\alpha > 1$.  
\begin{prop} \label{Prop:EigenFunction}
  Let $a(\alpha) \deq (1 + \alpha^{2k-n})/(1+\alpha^{-n})$.
  The function $\Phi$, defined for the path $h$ as 
  \begin{equation} \label{Eq:ToBeEst}
    \Phi(h) \deq
      \sum_{x=1}^{n-1} \Bigl( \alpha^{h(x)} - \alpha^{-x}a(\alpha)
        \Bigr) \sin(\pi x/n)  \,,
  \end{equation}
  is the second eigenfunction for the biased exclusion process,
  with eigenvalue
  \[
  1 - \frac{1 - 2\sqrt{pq}\cos(\pi/n)}{n-1} \,.
  \]
\end{prop}

We let $\theta = q/p$; note our convention is $\theta < 1$.  For 
a path $h$ and vertex $0 \leq i \leq n$,
let 
\begin{equation*}
  f_h(i) = \sum_{1 \leq j \leq i} \one\{h(j) - h(j-1) = 1\}
\end{equation*}
be the number of up-edges before $i$.  We have
$f_h(0) = 0$ and $f_h(n) = k$.

Define $g^\star_h(i) = \theta^{i-f_h(i)}$ for $i=0,1,\ldots,n$.

\begin{lem}
  Let $\tilde{h}^{(i)}$ be the path obtained by
  applying an update to $h$ at internal vertex $i$.
  Then 
  \begin{equation} \label{eq:parabolic}
    \E_h[g^\star_{\tilde{h}^{(i)}}(i)] 
      = q g^\star_{h}(i-1))+ p g^\star_{h}(i+1) \,.
  \end{equation} 
\end{lem}
\begin{proof}
  Consider the case where $i$ is a local extremum in $h$.
  If the path at $i$ is refreshed to a local maximum, then
  $f_{\tilde{h}^{(i)}}(i) = f_h(i-1) + 1$, while if the path
  is refreshed to a local minimum, then 
  $f_{\tilde{h}^{(i)}}(i) = f_h(i+1) - 1$.  Therefore,
  \begin{equation*}
    \E_h[ g^\star_{\tilde{h}^{(i)}}(i) ]
	 = q \theta^{i - (f_h(i-1) + 1)} + p \theta^{i - (f_h(i+1) -1)}
	 = q g^\star_h(i-1) + p g^\star_h(i+1) \,.
   \end{equation*}
   
   In the case where $h(i-1) < h(i) < h(i+1)$,
   the update at $i$ must leave the path unchanged.  In this
   case, $f_h(i-1) = f_h(i) - 1$ and
   $f_h(i+1) = f_h(i) + 1$.  Therefore,
   \begin{equation*}
     q g^\star_h(i-1) + p g^\star_h(i+1)
       = q \theta^{i - 1- (f_h(i) - 1)} + p
	\theta^{i+1 - (f_h(i) + 1)} 
       = g^\star_h(i) = \E_h[ g^\star_{\tilde{h}^{(i)}} ] \,.
   \end{equation*}
	
   Finally, suppose $h(i-1) > h(i) > h(i+1)$;
   again, the update at $i$ does not change the path.
   Since $f_h(i-1) = f_h(i) = f_h(i+1)$ in this case,
   \begin{equation*}
     q g^\star_h(i-1) + p g^\star_h(i+1)
       = q \theta^{(i-1) - f_h(i) } +
       p \theta^{(i+1) - f_h(i)}
       = (q \theta^{-1} + p\theta) g^\star_h(i) 
       = g^\star_h(i)\,.
   \end{equation*}
\end{proof}
For any constant $c$, the function $g_h(i) = g_h^\star(i) - c$ also satisfies
\begin{equation*}
  \E_h[ g_{\tilde{h}^{(i)}}(i) ] = q g_h(i-1) + p g_h(i+1) \,.
\end{equation*}
Define 
\begin{equation*}
  a(\theta) = \frac{1 + \theta^{n/2-k}}{1+\theta^{n/2}}
  = \frac{1 + \alpha^{2k-n}}{1 + \alpha^{-n}} \,,
\end{equation*}
and let
\begin{equation*}
  c(n,k,\theta) = \frac{1 + \theta^{n/2-k}}{1 + \theta^{-n/2}}
  = \theta^{n/2} \Bigl( \frac{1 + \theta^{n/2-k}}{1 + \theta^{n/2} }\Bigr) 
  = a(\theta) \theta^{n/2} \,.
\end{equation*}
Define
\[
g_h(i) = g_h^\star(i) - c(n,k,\theta) \,.
\]
\begin{proof}[Proof of Proposition \ref{Prop:EigenFunction}]
  Let $\phi:\{0,1,\ldots,n\} \to \R$ satisfy
  \begin{align*}
    \phi(0) & = 0, \quad \phi(n)  = 0 \\
    \lambda \phi(x) & = (p \phi(x-1) + q \phi(x+1))
                    & x = 1,\ldots,n-1 \,.
  \end{align*}
  That is, $\phi$ is the eigenfunction for the $q \uparrow, p
  \downarrow$ random  walk on $\{0,1,\ldots,n\}$ with absorbing states
  $0$ and $n$. A direct verification shows that 
  \begin{equation*} 
    \phi(x) = \theta^{-x/2} \sin(\pi x/n),
    \quad \lambda = 2\sqrt{pq} \cos(\pi/n)
  \end{equation*}
  is a solution.
  Note that
  \begin{align}
    \label{Eq:gphi}
    g_h(0) \phi(1) q + g_h(n) \phi(n-1) p 
    &  =  [1 - c] \theta^{-1/2} q \sin(\pi/n) \\
    & \quad + [\theta^{n-k} - c] \theta^{-n/2}\theta^{1/2}p 
      \sin(\pi - \pi/n) \nonumber \\
    & = \sqrt{pq} \sin(\pi/n)[1 + \theta^{n/2-k}
      - c[1 + \theta^{-n/2}]] \nonumber \\
    & = 0 \,. \nonumber
  \end{align}

  Define
  \begin{equation} \label{Eq:PhiDefn}
    \Phi(h) = 
    \sum_{x=1}^{n-1} g_h(x) \phi(x) \,.
  \end{equation}

  Let $\tilde{h}$ be the configuration obtained after one step of the chain
  when started from $h$; as before
  let $\tilde{h}^{(x)}$ be the update given that
  internal vertex $x$ is selected for an update.
  \begin{align*}
    \E_h[ \Phi(\tilde{h}) ]
    & = \sum_{x=1}^{n-1} \E_h[ g_{\tilde{h}}(x) ] \phi(x)\\
    & = \sum_{x=1}^{n-1} \Bigr[ \Bigl(1 - \frac{1}{n-1}\Bigr)
      g_h(x) + \frac{1}{n-1}\E_h[ g_{\tilde{h}^{(x)}} ]
      \Bigr] \phi(x) \\
    & = \Bigl( 1 - \frac{1}{n-1} \Bigr)  \Phi(h) + 
      \frac{1}{n-1}\sum_{x=1}^{n-1} [q g_h(x-1) + p g_h(x+1)] \phi(x) 
  \end{align*}
  The sum on the right equals
  \begin{multline*}
    \sum_{x=1}^{n-1} g_h(x)[q \phi(x+1) + p \phi(x-1) ] 
    + [ g_h(0) \phi(1) q + g_h(n) \phi(n-1) p ] \\
    = \lambda \sum_{x=1}^{n-1} g_h(x) \phi(x) =
    \lambda \Phi(h) \,,
  \end{multline*}
  by \eqref{Eq:gphi}.  Therefore,
  \[
  \E_h[ \Phi(\tilde{h}) ] = \Bigl( 1 - \frac{1 - \lambda}{n-1} \Bigr) \Phi(h)
  \]
  Note that $\phi(x) > 0$ for $x=1,\ldots,n-1$, and $g_h$ is increasing in $h$,
  so $\Phi$ is increasing.  An increasing eigenfunction always corresponds to
  the second eigenvalue, so it must be the one with largest (non unity)
  eigenvalue.  The second largest eigenvalue equals
  \[
  1 - \frac{1 - 2\sqrt{pq}\cos(\pi/n)}{n-1} \,.
  \]
  Note that $h(x) = 2f_h(x) - x$, so we have
  \begin{align*}
    \Phi(h) & = \sum_{x=1}^{n-1} g_h(x) \phi(x) \\
            & = 
              \sum_{x=1}^{n-1} \Bigl[ \theta^{x-f_h(x)} - c(n,k,\theta) \Bigr]
              \theta^{-x/2} \sin(\pi x/n) \\
            & = \sum_{x=1}^{n-1} \Bigl[ \alpha^{h(x)} - \theta^{(n-x)/2}
              \frac{1 + \theta^{n/2-k}}{1 + \theta^{n/2}} \Bigr] \sin(\pi x/n) \\
            & = \sum_{x=1}^{n-1} \alpha^{h(x)} \sin(\pi x/n) -
              \xi(n,k,\alpha)\,.
  \end{align*}
  Let
  \[
  \Psi(h) \deq \sum_{x=1}^{n-1} \alpha^{h(x)} \sin(\pi x/n)\,.
  \]
  Since $\xi(n,k,\alpha)$ does not depend on $h$, and
  the eigenfunction $\Phi$ must be orthogonal to the constants,
  it follows that $\xi(n,k,\alpha) = E_\pi(\Psi)$.
  Since $\sin(\pi(n-x)/n) = \sin(\pi x/n)$,
  \begin{equation*}
    E_\pi \Psi   = a(\theta)  \sum_{x=1}^{n-1} \theta^{(n-x)/2}
    \sin(\pi x/n) 
    = a(\theta) \sum_{x=1}^{n-1} \alpha^{-x} \sin(\pi x/n) \,.
  \end{equation*}
\end{proof}

To apply Wilson's Lower Bound, we need to bound $\max_h \Phi(h)$ from
below, and $R := | (\Phi(\tilde{h}) - \Phi(h))|^2$ from above.
Define
\begin{equation} \label{Eq:h0def}
  h_0(x) = 
  \begin{cases}
    x & x \leq k \\
    2k - x & k < x \leq n \,.
  \end{cases}
\end{equation}

\begin{lem}
  For $h_0$ defined in \eqref{Eq:h0def},
  \begin{equation} \label{Eq:Phih0}
    \begin{split}
      \Phi(h_0) & =
      \sum_{x=1}^k \alpha^{x}  (1 - \alpha^{-2x}) a(\alpha) \sin(\pi x/n) \\
      & \quad + \sum_{x=k+1}^{n/2} \alpha^x 
      \Bigr( \frac{(\alpha^{2k}-1)(\alpha^{-2x}+\alpha^{-n})}{1+\alpha^{-n}}
      \Bigr)
      \sin(\pi x/n)  \,.
    \end{split}
  \end{equation}
\end{lem}
\begin{proof}
  Using that $\sin(\pi x/n) = \sin(\pi (n-x)/n)$, 
  we pair together the terms at $x$ and $n-x$ in \eqref{Eq:ToBeEst} so that
  \begin{align*}
    \Phi(h_0) & = \sum_{x=1}^{k} \left( \alpha^x + \alpha^{2k-n + x} -
                a(\alpha)(\alpha^{-x} + \alpha^{x-n}) \right) \sin(x\pi/n)\\
              & \quad
                + \sum_{x=k}^{n/2} \left( \alpha^{2k-x} + \alpha^{2k-n + x} -
                a(\alpha)(\alpha^{-x} + \alpha^{x-n}) \right) \sin(x\pi/n) \,.
  \end{align*}
  The first sum simplifies to
  \[
  \sum_{x=1}^k \alpha^{x}  (1 - \alpha^{-2x}) 
  \Bigl( \frac{1+\alpha^{2k-n}}{1+\alpha^{-n}} \Bigr) \sin(\pi x/n)\,,
  \]
  and the second to
  \[
  \sum_{x=k+1}^{n/2} \alpha^x 
  \Bigr( \frac{(\alpha^{2k}-1)(\alpha^{-2x}+\alpha^{-n})}{1+\alpha^{-n}}
  \Bigr)
  \sin(\pi x/n) \,.
  \]
\end{proof}

\begin{lem} \label{Lem:EFRB}
  Let $h_0$ be as in \eqref{Eq:h0def}, and for a path $h$, let
  $\tilde{h}$ be one step of the exclusion chain started from $h$.
  Let $\gamma = 1 - \lambda$ be the spectral gap.
  Define
  \[
  R \deq \max_h |\Phi(\tilde{h}) - \Phi(h)|^2 \,.
  \]
  If $0 < n\beta \leq \log n$, then
  \[
  \log\Bigl( \frac{\gamma \Phi(h_0)^2}{2R} \Bigr) \geq [1 + o(1)] \log n
  \,.
  \]
\end{lem}
\begin{proof}
  Fix $b < k$.  From \eqref{Eq:Phih0},
  \begin{align}
    \Phi(h_0) & \geq 
                \frac{\sin(\pi b/n)}{2} \sum_{x=b}^k \alpha^x(1 - \alpha^{-2x})
                \nonumber \\
              & = \frac{\sin(\pi b/n)}{2} \alpha^k \frac{(\alpha - \alpha^{-(k-b)})(1 -
                \alpha^{-(b+k)})}{\alpha - 1} \label{Eq:Phih0B}\,.
  \end{align}
  If $\tilde{h}$ is obtained by a single update to $h$ at $x$,
  the $|\tilde{h}(x) - h(x)| \leq 2$, and
  \[
  |\alpha^{h(x)} - \alpha^{\tilde{h}(x)}|
  \leq 2\alpha^{k} \log(\alpha) \,.
  \]
  Thus, if $R = \max_h |\Phi(\tilde{h}) - \Phi(h)|^2$, then
  \begin{equation} \label{Eq:R0}
    \sqrt{R} \leq 2\alpha^k (\alpha - 1) \,.
  \end{equation}
  Letting $b = k/2$ so that $b/n \to \rho/2$, 
  equations \eqref{Eq:Phih0B} and \eqref{Eq:R0} show that
  \begin{equation} \label{Eq:Phih0C}
    \frac{\Phi(h_0)^2}{2R}
    \geq c_0 \left[\frac{(\alpha -
        \alpha^{-k/2})(1-\alpha^{-3k/2})}{
        (\alpha - 1)^2} \right]^2 \,.
  \end{equation}
  The spectral gap $1-\lambda = \gamma$ satisfies
  \begin{align}
    \gamma & = \frac{1 - 2\sqrt{pq}\cos(\pi/n)}{n-1} \nonumber \\
           & =  \frac{\beta^2/2 + O(\beta^4) + \frac{\pi^2}{2n^2} +
             O(n^{-4})}{n-1} \label{Eq:SG} \,.
  \end{align}
  Suppose that $n^{-1} \leq \beta \leq \frac{\log n}{n}$.  Then
  from \eqref{Eq:Phih0C} and \eqref{Eq:SG} we have
  \[
  \log\Bigl(\frac{\gamma\Phi(h_0)^2}{2R}\Bigr) \geq 
  \log \Bigl( c_1 \frac{n}{\log^4 n} \Bigr)
  = [1 + o(1)]\log n \,.
  \]
  
  If $n\beta \to \zeta$, where $0 \leq \zeta \leq 1$, then 
  \[
  \liminf_{n \to \infty} \frac{\gamma \Phi(h_0)^2}{n 2R}
  \geq 
  \begin{cases}
    c_0 \left[\frac{(1 - e^{-\zeta \rho/2})(1 - e^{-3\zeta
          \rho/2})}{\zeta ^2}\right]^2 & \zeta > 0 \\
    c_0 \left( \frac{3\rho^2}{4} \right)^2 & \zeta = 0
  \end{cases}
  \,.
  \]
  The right-hand side is bounded below for $0 \leq \zeta \leq 1$, so we
  conclude that
  \[
  \log\Bigl( \frac{\gamma \Phi(h_0)^2}{2R} \Bigr) \geq [1 + o(1)] \log n \,.
  \]
\end{proof}

\begin{prop} \label{Prop:LB1}
  If $n \beta \to \zeta$ where $0 \leq \zeta$, then
  \begin{equation} \label{Eq:nb0}
    \tmix(\ep) \geq \frac{n^3}{\pi^2+\zeta^2}[1 + o(1)]
    \Bigl(\log n + \log[(1-\ep)/\ep)] \Bigr) \,.
  \end{equation}
\end{prop}
\begin{proof}
  From \eqref{Eq:SG},
  the spectral gap $1-\lambda = \gamma$ satisfies
  \[
  \gamma = \frac{\pi^2 + \zeta^2}{2 n^3}[1 + o(1)] \,.
  \]
  Using Lemma \ref{Lem:EFRB} in  \ocite{W:MTS} (see also Theorem 13.5 of \fullocite{LPW} for a discussion) yields
  \begin{align}
    \tmix(\ep) & \geq \frac{1}{2\log(1/\lambda)}\left[
                 \log\left( \frac{(1-\lambda)\Phi(x)^2}{2R}\right) + \log((1-\ep)/\ep)
                 \right] 
                 \label{Eq:WilsonLB} \\
               & = \frac{n^3}{( \pi^2 + \zeta^2)}[1 + o(1)]\Bigl(\log n  +
                 \log[(1-\ep)/\ep] \Bigr)\,, \nonumber
  \end{align}
  which yields \eqref{Eq:nb0}.  Note that this matches the lower bound
  in Theorem 4 of Wilson (2004) for the symmetric exclusion when $\lim_n \beta n = 0$.
%  The proof where $n\beta \to \zeta$ is the same.
\end{proof}

\begin{prop} \label{Prop:LB2}
If $n\beta \to \infty$ but $n\beta \leq \log n$, then
\[
\tmix(\ep) \geq \frac{n}{\beta^2}[1 + o(1)](\log n + \log[(1-\ep)/\ep]) \,.
\]
\end{prop}
\begin{proof}
This again follows from \eqref{Eq:SG}, \eqref{Eq:WilsonLB} and Lemma \ref{Lem:EFRB}.
\end{proof}

\section{Upper Bounds}
\label{Sec:UB}

\subsection{Nearly unbiased}
\begin{prop} \label{Prop:NearUn}
There exists a constant $c_1$ such that
if $n \beta \leq 1$, then 
\[
\tmix(\ep) \leq c_1 n^3 \log n \,.
\]
\end{prop}
\begin{proof}
We now define a Markov chain $(\sigma_t, \eta_t)$ so that
\begin{itemize}
\item $\sigma_t$ and $\eta_t$ are \emph{labelled} $k$-particle configurations,
\item if the labels are erased, $(\sigma_t)$ and $(\eta_t)$ each are
biased exclusion processes.
\end{itemize}

We say a labelled particle is \emph{coupled} at time $t$ if it occupies the same
vertex in both $\sigma_t$ and $\eta_t$.   

We now describe a move of this chain from state $(\sigma, \eta)$: 
Pick an edge $e$ among the $n-1$ edges uniformly at random.  We consider 
several cases.

\begin{itemize}
\item \emph{Both $\sigma$ and $\eta$ have no particles on $e$.}  The chain remains
at $(\sigma, \eta)$.
\item \emph{One of $\sigma, \eta$ contains two particles on $e$, and one of
$\sigma, \eta$ contains one particle on $e$.}   Suppose, without loss of
generality, that $\sigma$ contains one particle on $e$.   Toss a $p$-coin to
determine where the particle is placed in $\sigma$.   If the single particle on $e$ in $\sigma$
is coupled, or has the same label as one of the particles on $e$ in $\eta$, arrange 
the two particles on $e$ in $\eta$ to preserve or facilitate the coupling.
Otherwise, toss a fair coin to determine the placement of the two particles in
$\eta$.
\item \emph{Both $\sigma$ and $\eta$ have two particles on $e$}.  
Toss a fair coin to determine the placement of the two particles on $e$ in $\sigma$.
Place the particles in $\eta$ on $e$ to preserve or facilitate any couplings; if
no coupling is possible, toss a fair coin to determine the particle placement on $e$.
\end{itemize}

The distance $D_i(t)$ between particle $i$ in $\sigma$ and particle $i$ in 
$\eta$ performs a delayed nearest-neighbor walk, with possible bias $\beta$ at each move 
(sometimes the bias is to the right, sometimes to the left).  The probability it moves
is at least $1/(n-1)$.    We can thus couple it to a random walk  $(S_t)$ with constant
upward bias $\beta$ so that $D_i(t) \leq S_t$ until $D_i(t)$ hits zero.

Consider the biased random walk $(S_t)$ on $\Z$ with positive bias $\beta$,
holding probability $1-\frac{1}{n-1}$, and $S_0 = n$; if
\[
\tau = \min\{t \geq 0 \,:\, S_t = 0 \} \,, \quad \text{ and } 
\tau_i = \min\{t \geq 0 \,:\, D_i(t) = 0 \} \,,
\]
then 
\[
\P( \tau_i > u) \leq \P( \tau > u ) \,.
\]
We have
\begin{equation*}
\P(\tau \leq t) \geq \P_{n} (S_t \leq 0 ) \\
 = \P\left( Z_t \leq \frac{-n - t \beta/(n-1)}{\sqrt{4tpq/(n-1)}} \right)
\end{equation*}
where $Z_t = \frac{S_t - \E_n(S_t)}{\var(S_t)}$.  
By the Central Limit Theorem, since $\beta n \leq 1$, 
there is a constant $c_0 > 0$ such that, for $n$ large enough,
\[
\P_n(S_{n^3} \leq 0) \geq c_0 \,.
\]
Thus by taking $c_1$ large enough,
\[
\P_n(\tau > c_1 n^3) \leq (1-c_0)^{c_1} < \frac{1}{2} \,.
\]

%
%
%Let $M_t = S_t - (t\beta)$; the process $(M_t)$ is a martingale, and
%\[
%S_t =  M_t + t \beta
%\]
%where $M_t$ is a martingale.  Note that $\var(M_{t+1} \mid \F_t) \geq 1/5$.
%If $N_t = M_t + 2\kappa_1 n$ and $\tau = \min\{ t \st N_t \leq 0\}$, then
%$N_{t \wedge \tau}$ is a non-negative martingale.   Applying Proposition 17.19
%of \fullocite{LPW} shows that
%\begin{equation*}
%  \P( \tau > s n^2) \leq \frac{20 (1+2\kappa_1)}{\sqrt{s}} \,.
%\end{equation*}
%Pick $s$ so that this probability is at most $1/8$.  Set 
%$\tau' = \min\{ t \st S_t \leq 0 \}$.
%Note that if $\tau \leq s n^2$, then for some $t \leq s n^2$,
%\begin{equation*}
%  0 \geq M_t + 2\kappa_1 n 
%  = S_t - t\beta + 2\kappa_1 n \geq S_t - \beta s n^2 + 2\kappa_1 n
%  = S_t + n(2\kappa_1 - s \beta n) \,,
%\end{equation*}
%whence $S_t \leq -n\kappa_1$ and we must have $\tau' \leq t \leq sn^2$.
%We conclude that
%\begin{equation*}
%  \P(\tau' > sn^2)  \leq \P(\tau > sn^2) \leq \frac{1}{8} \,.
%\end{equation*}
%By adding a delay with probability $1/(n-1)$, if $\tau_i \deq \min\{t \st D_i(t) = 0\}$,
%we have
%\begin{equation*}
%  \P(\tau_i > s n^2) \leq \frac{1}{2} \,.
%\end{equation*}
If we run $2 \log_2 n$ blocks of $c_1 n^3$ moves, then we have
\begin{equation*}
  \P( \tau_i > 2c_1 n^3 \log_2 n) \leq \frac{1}{n^2} \,.
\end{equation*}
Setting $\tau_{{\rm couple}} \deq \min\{t \geq 0 \;:\; \sigma_t = \eta_t\}$, 
\[
\P\bigl( \tau_{{\rm couple}} > 2c_1 n^3\log_2 n \bigr) 
\leq \sum_{i=1}^k \P( \tau_i > 2c_1 n^3\log_2 n )< \frac{1}{n} \,.
\]
If $d(t) = \sup_h\| P^t(h,\cdot) - \pi\|_{{\rm TV}}$, then
 $d(2c_1n^3\log_2 n) \leq \frac{1}{n}$, and
\[
\tmix(\ep) \leq 2c_1 n^3\log_2 n
\]
for $n$ large enough.
\end{proof}

%
% Consider $n$ particles placed on a path of length $2n$ so that no
% two particles occupy the same vertex.  Thus the space of
% configurations is
% \[
% \X = \left\{ \sigma \in \{0,1\}^{2n} \st \sum_{k=1}^{2n} \sigma(k) =
%   n \right\} \,.
% \]
% (If $\sigma(k) = 1$, we say that site $k$ is \emph{occupied}, while
% if $\sigma(k) = 0$, site $k$ is \emph{unoccupied}.)  There is a
% bijection between $\X$ and
% \begin{align*}
%  \X' & = \{ f:\{1,2,\ldots,2n\} \to \Z \st
%   f(0) = f(2n) = 0, \\
%   & \qquad \text{ and } f(k+1) = f(k) \pm 1, \; k=0,1,\ldots,n-1 \},
%\end{align*}
%the space of nearest-neighbor walks on $\Z$ of length $2n$ which
% start and end at $0$.  The bijective map $\sigma \mapsto f$ is given
% by $f(k) = f(k-1) - (-1)^{\sigma(k)}$.  That is, working
% left-to-right, occupied (unoccupied) sites are in correspondence
% with ``up'' (``down'') moves of the walk.
%
% We consider configurations $\sigma$ and $\sigma'$ to be adjacent if
% $\sigma'$ can be obtained by taking a particle and moving it to an
% adjacent unoccupied site.  In the nearest-neighbor walk
% representation, moving a particle to the right corresponds to
% changing a local maximum (i.e., an ``up-down'') to a local minimum
% (i.e. a ``down-up'').  Moving a particle to the left changes a local
% minimum to a local maximum.  See Figure~\ref{Fig:ExclMov}.
%

\subsection{Path coupling}

\begin{figure}[h]
  \begin{center}
    \includegraphics[scale=0.45]{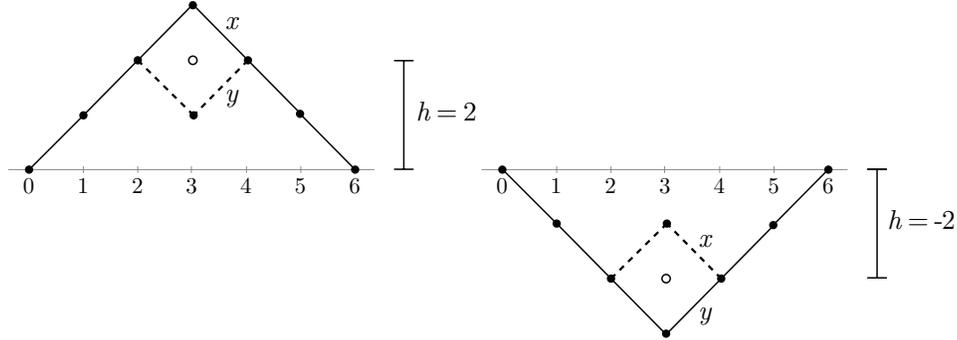}
    \caption{Neighboring configurations $x$ and
      $y$. \label{Fig:WalkNeighbors}}
  \end{center}
\end{figure}

We consider configurations $x$ and $y$ to be adjacent if $y$ can be
obtained from $x$ by taking a particle and moving it to an adjacent
unoccupied site.  In the path representation, moving a particle to the
right corresponds to changing a local maximum (i.e., an ``up-down'')
to a local minimum (i.e. a ``down-up'').  Moving a particle to the
left changes a local minimum to a local maximum.  See
Figure~\ref{Fig:ExclMov}, where $v = 3$.

\begin{thm} \label{thm:mixbe} Consider the biased exclusion process
  with bias $\beta = \beta_n = 2p_n -1 > 0$ on the segment of length
  $n$ and with $k$ particles.  Set $\alpha = \sqrt{p_n/(1-p_n)}$.
  For $\ep > 0$, if $n$ is large enough, then
  \[
  \tmix(\ep) \leq \frac{2n}{\beta^2} \left[ \log(1/\ep) + \log\left[
      \frac{\alpha(\alpha^k -1)(\alpha^{n-k} - 1)}{(\alpha -1 )^2} \right]
  \right] \,.
  \]
  In particular, if $\beta \leq {\rm const.} < 1$, then
  $\alpha = 1 + \beta + O(\beta^2)$, so
  \[
  \tmix(\ep) \leq \frac{2n}{\beta^2}\Bigl[ \log(\ep^{-1}) +
  n[\beta + O(\beta^2)] -2 \log \beta + O(\beta) \Bigr] \,.
  \]
\end{thm}

\begin{rmk}
  Note that whenever
  $c_1 (\log n)/n < \beta < c_2< 1$ for constants $c_1$ and $c_2$,  the ratio
  of the upper and lower bounds is bounded.  Thus there is a pre
  cut-off for this chain in this regime.
\end{rmk}

\begin{proof}
%  \emph{Path Coupling Upper bound}.

 % Suppose that $y$ can be obtained from $x$ by a single move replacing
  % a local max (min) with a local min (max).

  % Set $\alpha \deq \sqrt{p/q} > 1$, where $q \deq 1-p$, and
  For $\alpha = \sqrt{p/q} > 1$, define the distance between two configurations $x$
  and $y$ which differ by a single transition to be
  \begin{equation*}
    \ell(x, y) = \alpha^{n - k + h},
  \end{equation*}
  where $h$ is the height of the midpoint of the diamond that is
  removed or added. (See Figure \ref{Fig:WalkNeighbors}.)  Note that
  $\alpha > 1$ and $h \geq -(n-k)$ guarantee that $\ell(x, y) \geq 1$, so
  we can use path coupling -- see, e.g., Theorem 14.6 of \fullocite{LPW}.  We again let $\rho$
  denote the path metric on $\X$ corresponding to $\ell$.
  %, as defined
%  in \eqref{Eq:PathMetricDefn}.

  % If it is not a local extremum, the chain remains in its current
  % configuration.  If it is a local maximum, this maximum is inverted
  % with probability $p$, and the chain remains where it is with
  % probability $1-p$.  If it is a local minimum, the minumum is
  % inverted to a local maximum with probability $1-p$, and the chain
  % remains where it is with probability $p$.

  We couple from a pair of initial configurations $x$ and $y$ which
  differ at a single vertex $v$ as follows: choose the same vertex in
  both configurations, and propose a local maximum with probability
  $1-p$ and a local minimum with probability $p$.  For both $x$ and
  $y$, if the current vertex $v$ is a local extremum, refresh it with
  the proposed extremum; otherwise, remain at the current state.

  Let $(X_1,Y_1)$ be the state after one step of this coupling.  There
  are several cases to consider.

  The first case is shown in Figure \ref{Fig:WalkNeighbors}.  Let $x$
  be the upper configuration, and $y$ the lower.  Here the edge
  between $v-2$ and $v-1$ is ``up'', while the edge between $v+1$ and
  $v+2$ is ``down'', in both $x$ and $y$.  If $v$ is selected, the
  distance decreases by $\alpha^{n-k+h}$.  If either $v-1$ or $v+1$ is
  selected, and a local minimum is selected, then the lower
  configuration $y$ is changed, while the upper configuration $x$
  remains unchanged.  Thus the distance increases by $\alpha^{n-k+h-1}$
  in that case.  We conclude that
  \begin{align}
    \E_{x,y}[ \rho(X_1, Y_1) ] - \rho(x,y)
    & = -\frac{1}{n-1}\alpha^{h+n-k} + \frac{2}{n-1}p \alpha^{h+n-k -1}
      \nonumber \\
    & = \frac{\alpha^{h+n-k}}{n-1} \left( \frac{2p}{\alpha} - 1 \right) 
    = \frac{\alpha^{h+n-k}}{n-1}\left( 2\sqrt{pq} - 1 \right) \,.
      \label{Eq:AEPC1} 
      % & = \frac{\alpha^{h+n}}{2n-1} \left(2 \sqrt{pq} - 1 \right)
      % \,. \label{Eq:AEPC1}
  \end{align}
  In the case where $x$ and $y$ at $v-2, v-1, v, v+1, v+2$ are as in
  the right panel of Figure \ref{Fig:WalkNeighbors}, we obtain
  \begin{align}
    \E_{x,y}[ \rho(X_1, Y_1) ] - \rho(x,y)
    & = -\frac{1}{n-1}\alpha^{h+n-k} + \frac{2}{n-1}(1-p) \alpha^{h+n+1} 
      \nonumber \\
    & = \frac{\alpha^{h+n-k}}{n-1}\left( 2\alpha (1-p) - 1 \right) 
    = \frac{\alpha^{h+n-k}}{n-1}\left(2 \sqrt{pq} - 1 \right)
      \label{Eq:AEPC2}
      % & = \frac{\alpha^{h+n}}{2n-1} \left(2 \sqrt{pq} - 1 \right)
      % \,.
  \end{align}
  (We create an additional disagreement at height $h+1$ if either $v-1$ or
  $v+1$ is selected and a local maximum is proposed; the top
  configuration can accept the proposal, while the bottom one rejects
  it.) 
%   To obtain as large a uniform contraction as possible, we set
%  the right-hand sides of \eqref{Eq:AEPC1} and \eqref{Eq:AEPC2} equal
%  to one another and solve for $\alpha$.  This yields
%  \[
%  \alpha = \sqrt{\frac{p}{q}} = \sqrt{\frac{1+\beta}{1-\beta}},
%  \]
%  where $\beta$ is the \emph{bias}
%  $\beta = p - q$.
  Since $p > 1/2$, we have $\delta \deq 1 - 2\sqrt{pq} > 0$,
  and both \eqref{Eq:AEPC1} and  \eqref{Eq:AEPC2} reduce to
  \begin{equation} \label{Eq:AEPC3} \E_{x,y}[ \rho(X_1, Y_1) ] -
    \rho(x,y) = - \frac{\alpha^{h+n-k}}{n-1} \delta\,.
  \end{equation}
  Now consider the case on the left of Figure \ref{Fig:MoreConfigs}.
  We have
  \begin{align*}
    \E_{x,y}[ \rho(X_1, Y_1) ] - \rho(x,y)
    & = -\frac{1}{n-1}\alpha^{h+n-k} 
      + \frac{1}{n-1}q \alpha^{h+n-k+1} 
      + \frac{1}{n-1}p \alpha^{h+n-k-1} 
      \nonumber \\
    & = \frac{\alpha^{h+n-k}}{n-1}\left(
      q\alpha + \frac{p}{\alpha} - 1 \right)
      \nonumber \\
    & = -\frac{\alpha^{h+n-k}}{n-1} \delta \,,
  \end{align*}
  which gives again the same expected decrease as \eqref{Eq:AEPC3}.
  (In this case, a local max proposed at $v-1$ will be accepted only
  by the top configuration, and a local min proposed at $v+1$ will be
  accepted only by the bottom configuration.)  The case on the right
  of Figure \ref{Fig:MoreConfigs} is the same.

  Thus, \eqref{Eq:AEPC3} holds in all cases.  That is, since
  $\rho(x,y) = \ell(x,y) = \alpha^{h+n-k}$,
  \begin{equation*}
    \E_{x,y}[\rho(X_1, Y_1)]
    = \rho(x,y)\left(1 - \frac{\delta}{n-1} \right)
    \leq \rho(x,y) e^{-\frac{\delta}{n-1}} \,.
  \end{equation*}

  The diameter of the state-space is the distance from the configuration with $k$ ``up''
  edges followed by $n-k$ ``down'' edges to the configuration with $n-k$
  ``down edges'' followed by $k$ ``up'' edges.
  To move from the former to the latter, first flip the top-most maxima,
  next the subsequent two maxima, continuing down $k-1$ levels.
  At level $j$, there are $j$ maxima to flip.  Each of the next 
  $n-2k+1$ levels will have $k$ maxima to flip.  The number of
  maxima in the last $k-1$ levels decrease by a unit at each depth.
  Thus, the distance travelled equals
  \begin{align*}
  \sum_{j=1}^{k-1} j \alpha^{n-k+k-j} + \sum_{j=k}^{n-k} k \alpha^{n-k+k-j}
  + \sum_{j=n-k+1}^{n-1}(n-j) \alpha^{n-k+k - j}
  \\
  = \frac{\alpha(\alpha^k - 1)(\alpha^{n-k}-1)}{(\alpha-1)^2}
  \end{align*}
%  
%  \[
%  \alpha^n \sum_{j=1}^n j \alpha^{n-j} + \alpha^n \sum_{j=1}^{n-1}
%  (n-k-j) \alpha^{-j} = \alpha \left( \frac{\alpha^n - 1}{\alpha - 1}
%  \right)^2 \,.
%  \]
%  
 Since $\delta \geq \beta^2/2$, Corollary 14.7 of \fullocite{LPW} gives % \ref{Cor:PCMixing} yields
  \[
  \tmix(\ep) \leq \frac{2n}{\beta^2} \left[ \log(1/\ep) + \log\left[
      \frac{\alpha (\alpha^k -1)(\alpha^{n-k}-1)}{(\alpha-1)^2} \right]
  \right] \,.
  \]
  Note that $\alpha = 1 + \beta + O(\beta^2)$ as $\beta \to 0$, so
  \[
  \tmix(\ep) \leq \frac{2n}{\beta^2}\left[ \log(\ep^{-1}) +
    n[\beta + O(\beta^2)] -2 \log \beta + O(\beta) \right] \,.
  \]
  In particular, if $\beta = \frac{1}{n}$, then
  $\tmix(\ep) = O(n^3 \log n)$, which is the same order as the 
  mixing time in the symmetric case.

  % and
  % \[
  % \frac{2n\beta + 2\log(\beta^{-1})}{\frac{\beta^2}{4n}} =
  % \left(\frac{8n^2}{\beta} + \frac{8n}{\beta^2}\log\beta^{-1}\right)
  % \left(1 + O(\beta)\right) \,.
  % \]
  %% The first term dominates if and only if $\beta \geq (\log n)/n$.
  %% If $\beta = O(1/n)$, then $\tmix = O(n^3 \log n)$.
%%
%

  \begin{figure}
    \centerline{\includegraphics[scale=0.35]{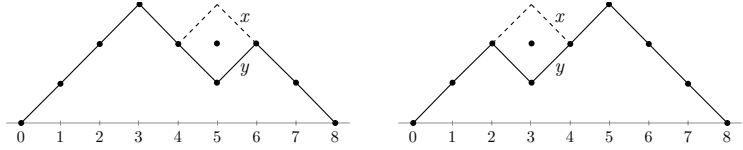}}
    \caption{More neighboring configurations. \label{Fig:MoreConfigs}}
  \end{figure}
  %% In continuous time chapter, give description in terms of Poisson
  %% clocks.  The results here imply that mixing time of the
  %% continuous-time chain is O(n) (Exercise, Define with individual
  %% clocks, then derive from discrete-time result that the mixing
  %% time is O(n)

  % The stationary measure $\pi$ can be described as the law of
  % $(V_1,\ldots,V_n)$, where $V_1,\ldots,V_n$ are samples drawn
  % without replacement from $\{-n, -n+1,\ldots,n-1,n\}$.

  % (This is relevant for the center-of-mass argument.)
\end{proof}

\section{Lower bound via a single particle}
\label{Sec:LBU}

\begin{prop} \label{Prop:LBU}
 Suppose that $n\beta \to \infty$.  For any $\ep > 0$ and
  $\delta > 0$, if $n$ is large enough, then
  \[
  \tmix(\ep) \geq \frac{(1-\delta)n^2}{2 \beta} \,.
  \]
 \end{prop}
 \begin{proof}
  %\emph{Lower Bound when $k\beta \to \infty$.} 
   We use the particle description here.  The
  stationary distribution is given by
  \begin{equation*}
    \pi(x) = \frac{1}{Z} \prod_{i=1}^k \left( \frac{p}{q}
    \right)^{z_i(x)}
    = \frac{1}{Z} (p/q)^{\sum_{i=1}^k z_i(x)},
  \end{equation*}
  where $(z_1(x), \ldots, z_k(x))$ are the locations of the $k$
  particles in the configuration $x$, and $Z$ is a normalizing
  constant. To see this, if $x'$ is obtained from $x$ by moving a
  particle from $j$ to $j+1$, then
  \[
  \frac{\pi(x)P(x,x')}{\pi(x')P(x',x)} = \frac{1}{(p/q)}
  \frac{\frac{1}{n-1}p}{\frac{1}{n-1}q} = 1 \,.
  \]
 % Similarly, the detailed balance equation is satisfied for a
  % transition moving one particle to the left.

  Let $L(x)$ be the location of the left-most particle of the
  configuration $x$, and let $R(x)$ be the location of the right-most
  unoccupied site of the configuration $x$.

  Let
  \[
  \X_{j,\ell} = \{ x \st L(x) =j, \; R(x) = \ell \} \,,
  \]
  and consider the transformation $T: \X_{j,\ell} \to \X$ which takes
  the particle at $j$ and moves it to $\ell$.  Note that $T$ is
  one-to-one on $\X_{j,\ell}$.

  We have
  \begin{equation*}
    \pi(\X_{j,\ell})\left( \frac{p}{q} \right)^{\ell - j}
    \leq \sum_{x \in \X_{j,\ell}} \pi(T(x))
    %= \sum_{x'} \sum_{\substack{x \in \X_{k,\ell}\\ T(x)
     %   = x'}} \pi(x')
    \leq 1 \,,
  \end{equation*}
  so
  \[
  \pi(\X_{j,\ell}) \leq \alpha^{-2(\ell-j)} \,.
  \]
  Letting $G = \{x \st L(x) \leq (1/2-b)n\}$, we have
%  Letting $G = \{x \st L(x) \leq (1-b)k\}$, we have
  \[
%  \pi(G) \leq \sum_{{j \leq (1-b)k,\; \ell \geq n-k}}\pi(\X_{j,\ell})
%  \leq k(n-k)\alpha^{-2 bk} \,.
   \pi(G) \leq \sum_{{j \leq (1/2-b)n,\; \ell \geq n/2}}\pi(\X_{j,\ell})
  \leq n^2 \alpha^{-bn} \,.
  \]

  We consider now starting from a configuration $x_0$ with
  $L(x_0) = bn/2$.%, where $1 > b > 0$.

  % Consider the transformation $x \mapsto T(x)$, where $T(x)$ is
  % obtained from $x$ by taking the left-most particle and moving it
  % to the right-most unoccupied site.  This is at most an $n$-to-$1$
  % map.  Thus, if $S$ is the set of configurations
  % \[
  % S \deq \{ x \st L(x) \leq n/2 \} \,,
  % \]
  % then
  % \begin{equation*}
  %   \pi(S)\left( \frac{p}{q} \right)^{n/2} \leq \sum_{x \in S}
  %   \pi(T(x)) \leq \sum_{\tilde{x}}
  %   \pi(\tilde{x})|T^{-1}(\tilde{x})| \leq n \,.
  % \end{equation*}
  % That is,
  % \begin{equation*}
  %   \pi(S) \leq n \alpha^{-n} \,.
  % \end{equation*}

  The trajectory of the left-most particle, $(L_t)$,
  % viewed only at the times when it moves, $(\tilde{L}_t)$,
  can be coupled with a delayed biased nearest-neighbor walk $(S_t)$
  on $\Z$, with $S_0 = bn/2$ and %$\{1,\ldots,2n\}$
  such that $L_t \leq S_t$, as long as $S_t > 1$.  The holding
  probability for $(S_t)$ equals $1 - \frac{1}{n-1}$.  By %comparison
  the gambler's ruin, the chance $S_t$ ever 
  reaches $1$ is bounded above by
  \[
  (q/p)^{bn/2} \leq e^{- \beta b n} \,.
  \]
  Therefore.
  \begin{equation} \label{Eq:LSComp} \P_{x_0}\{L_t > (1/2-b)n\} \leq
    e^{-\beta b n} + \P_{bn/2}\{ S_t > (1/2-b)n\} \,.
  \end{equation}

%
  % We now compare $\{Z_t\}$ to an \emph{unbounded} delayed biased
  % random walk.  For a process $(X_t)$, define the event
  % \[
  % A_u(X) \deq \{ X_u = X_0 + 1, \; X_j > X_0 \, \text{for } j =
  % (u+1),\ldots,t, \, \text{and } X_t > (1-b)n\} \,.
  % \]
  % Then $\{X_t > (1-b)n\} \subset \bigcup_{u=1}^t A_u(X)$, and if
  % $(S_t)$ is biased nearest-neighbor random walk on $\Z$ with
  % $S_0 = Z_0$, then $\P( A_u(Z) ) = \P( A_u(S) )$ and so
  % \begin{equation}
  %   \P\{Z_t > (1-b)n\} \leq \sum_{u=1}^t\P(A_u(Z)) =\sum_{s=1}^t
  %   \P(A_u(S)) \,. \label{Eq:ZS}
  % \end{equation}
  %% Since $S_0 = Z_0 = \tilde{L}_0 = L(x_0) = 0$,
  % Since
  % \[
  % \P(A_u(S)) \leq \P\{ S_t > (1-b)n\} \,,
  % \]
  %% Note that by translating the walk to $1$ at times $s$, since
  %% $S_0 \leq n/2$, we get $\P(A_s(S)) \leq \P\{S_t > n/2\}$ and
  % the bound \eqref{Eq:ZS}, along with the fact that that
  % $\tilde{L}_t \leq Z_t$, implies that
  % \[
  % \P\{L_t > (1-b)n\} \leq \P\{Z_t > (1-b)n\} \leq t\P\{S_t >
  % (1-b)n\} \,.
  % \]

  By Chebyshev's Inequality (recalling $S_0 = bn/2$),
  \[
  \P\{|S_t - bn/2 - \beta t/(n-1)| > M\} \leq \frac{\var(S_t)}{M^2}
  \leq \frac{t}{M^2(n-1)} \,.
  \]
  Taking $t_n = \frac{(1-4b)(n-1)n/2}{\beta}$ and $M = bn/2$ shows that
  \[
  \P_{bn/2}\{S_{t_n} > (1/2-b)n \} \leq \frac{4(1-4b)}{ b^2 \beta n} \to 0\,,
  \]
  as long as $\beta n \to \infty$.  Combining with \eqref{Eq:LSComp}
  shows that
  \[
  \P\{L_{t_n} > (1/2-b)n\} \leq e^{-b \beta n} + o(1)
  \,.
  \]

%
  % By Hoeffding's Inequality, taking $t_n = n/4\beta$,
  % \begin{equation*}
  %   \P\{S_{t_n} > n/2\} \leq e^{-n\beta/2} \,,
  % \end{equation*}
%
%
  % Therefore,
  % \begin{equation*}
  %   \P\{\tilde{L}_{t_n} > n/2\} \leq \frac{n}{4\beta} e^{-n\beta/2}
  %   \,.
  % \end{equation*}
%
%
  % Since $L_t$, the left-most particle, moves only with at most
  % probability $1/(2n-1)$, if $r_n = (2n-1)^2/16\beta$, then the
  % number of times $T(r_n)$ this particle moves is bounded above by a
  % Binomial($r_n,(2n-1)^{-1}$).  By Exercise \ref{Exer:BiCl},
  % \[
  % \P_{x_0}\{ T(r_n) > n/4\beta\} \leq e^{-(3-e)(2n-1)/16\beta} \,.
  % \]
%
  % Then
  % \begin{multline*}
  %   1 - \P_{x_0}\{ X_{r_n} \in G\} \leq \P_{x_0}\{ L_{r_n} > n/2, \;
  %   T(r_n) \leq n/4\beta\} +
  %   e^{-(3-e)(2n-1)/8\beta} \\
  %   \leq \P\{\tilde{L}_{t_n} > n/2\} + e^{-(3-e)(2n-1)/8\beta} \leq
  %   \frac{n}{4\beta} e^{-n \beta/2} + e^{-(3-e)(2n-1)/16\beta} \,.
  % \end{multline*}
  We conclude that as long as $\beta n \to \infty$,
  \[
  d(t_n) \geq \P_{x_0}\{X_{t_n} \in G\} - \pi(G) \geq 1 - o(1)
  \]
  as $n \to \infty$, whence
  $\tmix(\ep) \geq \frac{(1-4b)(n-1)n}{2 \beta}$ for sufficiently large $n$.

\end{proof}

\section*{Acknowledgements}

We thank Perla Sousi and Nayantara Bhatnagar  for helpful comments on an earlier version of
this paper.


\begin{bibdiv}
  \begin{biblist}
%    \bibselect{../BIB2E/sim}
    
\bib{A:RWG}{article}{
   author={Aldous, David},
  title =	 {Random walks on finite groups and rapidly mixing
                  Markov chains},
  conference =	 { title={Seminar on probability, XVII}, },
  book =	 { series={Lecture Notes in Math.}, volume={986},
                  publisher={Springer}, place={Berlin}, },
  date =	 {1983},
  pages =	 {243--297},
}

\bib{BBHM}{article}{
  author={Benjamini, Itai},
  author={Berger, Noam},
  author={Hoffman, Christopher},
  author={Mossel, Elchanan},
  title={Mixing times of the biased card shuffling and the asymmetric exclusion process},
  journal={Trans. Amer. Math. Soc.},
  volume={357},
  date={2005},
  number={8},
  pages={3013--3029 (electronic)},
}

\bib{GPR}{article}{
  author={Greenberg, S.},
  author={Pascoe, A.},
  author={Randall, D.},
  title={Sampling biased lattice configurations using exponential metrics},
  conference={ title = {ACM-SIAM Symposium on Discrete Algorithms}, address = {New York, New York}, year = {2009}, },
  year={2009},
  pages={76-85},
}
\bib{Lacoin}{article}{
   author={Lacoin, Hubert},
   title={Mixing time and cutoff for the adjacent transposition shuffle and
   the simple exclusion},
   journal={Ann. Probab.},
   volume={44},
   date={2016},
   number={2},
   pages={1426--1487},
}


    \bib{LPW}{book}{
   author={Levin, David A.},
   author={Peres, Yuval},
   author={Wilmer, Elizabeth L.},
   title={Markov chains and mixing times},
   note={With a chapter by James G. Propp and David B. Wilson},
   publisher={American Mathematical Society, Providence, RI},
   date={2009},
   pages={xviii+371},
}


\bib{W:MTS}{article}{
   author={Wilson, David Bruce},
   title={Mixing times of Lozenge tiling and card shuffling Markov chains},
   journal={Ann. Appl. Probab.},
   volume={14},
   date={2004},
   number={1},
   pages={274--325},
}


  \end{biblist}
\end{bibdiv}
\end{document}